\documentclass[12pt]{article}
\usepackage{amsmath,amsfonts,latexsym,amsthm,amssymb,mathabx}
\newtheorem{Theorem}{Theorem}[section]
\newtheorem{Lemma}[Theorem]{Lemma}
\newtheorem{Corollary}[Theorem]{Corollary}

\newtheorem{Remark}[Theorem]{Remark}

\theoremstyle{definition}

\usepackage{amssymb}

\newcommand\C{{\mathbb C}}
\newcommand\R{{\mathbb R}}
\newcommand\X{{\R^d}}
\newcommand\Q{{\mathbb Q_p}}
\newcommand\N{{\mathbb N}}
\newcommand\M{{\mathcal M}}

\newcommand\B{{\mathcal B}}

\newcommand\Aff{{\mathrm{ \mathbb{A} ff} (\Q)}}

\newcommand\La{\Lambda}

\newcommand\Ga{\Gamma}
\newcommand\ga{\gamma}

\newcommand{\PC}{{\mathcal P}_{cyl}}
\newcommand{\CF}{\mathcal F}

\newcommand{\D}{\mathcal D}

\DeclareMathOperator{\supp}{supp}

\begin{document}

\pagestyle{plain}
\title{Representations of the Infinite-Dimensional $p$-Adic Affine Group}

\date{}
\author{ \textbf{Anatoly N.
Kochubei}\\
Institute of Mathematics,\\
National Academy of Sciences of Ukraine, \\
Tereshchenkivska 3, \\
Kyiv, 01024 Ukraine\\
Email: kochubei@imath.kiev.ua \and \textbf{Yuri Kondratiev}\\
Department of Mathematics, University of Bielefeld, \\
D-33615 Bielefeld, Germany,\\
Dragomanov National Pedagogic University, Kyiv, Ukraine,\\
Email: kondrat@math.uni-bielefeld.de}

\maketitle

\vspace*{3cm}
\begin{abstract}
We introduce an infinite-dimensional $p$-adic affine group and construct its irreducible unitary representation. Our approach follows the one used by Vershik, Gelfand and Graev for the diffeomorphism group, but with modifications made necessary by the fact that the group does not act on the phase space. However it is possible to define its action on some classes of functions.
\end{abstract}
\vspace{2cm}
{\bf Key words: }$p$-adic numbers; affine group; configurations; Poisson measure; ergodicity

\medskip
{\bf MSC 2010}. Primary: 22E66. Secondary: 60B15.

\medskip
Corresponding author - A. Kochubei

\newpage
\section{Introduction}

Given a vector space $V$  the affine group  can be described concretely as the semidirect product of $V$ by $\mathrm{GL}(V)$, the general linear group of $V$:
$$
 \mathrm{Aff} (V)=V  \rtimes \mathrm{ GL} (V).
 $$
The action of $\mathrm{GL}(V)$ on $V$ is the natural one (linear transformations are automorphisms), so this defines a semidirect product.

Affine groups play important role in the geometry and its applications, see, e.g., \cite{Ar,Ly}. Several recent papers \cite{AJO,AK,EH,GJ,Jo,Ze} are devoted to representations of the real, complex and $p$-adic affine groups and their generalizations, as well as diverse applications, from wavelets and Toeplitz operators to non-Abelian pseudo-differential operators and $p$-adic quantum groups.

In the particular case of $p$-adic  field $V= \Q$ the group $\mathrm{Aff}(\Q)$ consists of pairs
$(a,b)$,  $a,b\in \Q$ with the group operation
$$
(a_2,b_2) (a_1, b_1) = (a_1 a_2, b_2 +a_2 b_1).
$$
We would like to extend this object to an infinite-dimensional group
$\Aff$ by using instead  of numbers $a,b$ functions on $\Q$ from an appropriate class.
Our aim is to construct irreducible representations of $\Aff$. As a rule, only special classes of irreducible representations can be constructed for infinite-dimensional groups. For various classes of such groups, special tools were invented; see \cite{Is,Ko} and references therein.

We will follow an approach by Vershik-Gefand -Graev \cite{VGG75} proposed in the case of
the group of diffeomorphisms.  A direct application of this approach meets certain difficulties
related with the absence of the possibility to define the action of the group $\Aff$ on a phase space
similar to \cite{VGG75}.  A method to overcome this problem is the main technical step in the present paper.

\section{$p$-Adic Numbers \cite{VVZ}}

Let $p$ be a prime
number. The field of $p$-adic numbers is the completion $\mathbb Q_p$ of the field $\mathbb Q$
of rational numbers, with respect to the absolute value $|x|_p$
defined by setting $|0|_p=0$,
$$
|x|_p=p^{-\nu }\ \mbox{if }x=p^\nu \frac{m}n,
$$
where $\nu ,m,n\in \mathbb Z$, and $m,n$ are prime to $p$. $\Q$ is a locally compact topological field. By
Ostrowski's theorem there are no absolute values on $\mathbb Q$, which are not equivalent to the ``Euclidean'' one,
or one of $|\cdot |_p$.

The absolute value $|x|_p$, $x\in \mathbb Q_p$, has the following properties:
\begin{gather*}
|x|_p=0\ \mbox{if and only if }x=0;\\
|xy|_p=|x|_p\cdot |y|_p;\\
|x+y|_p\le \max (|x|_p,|y|_p).
\end{gather*}

The latter property called the ultra-metric inequality (or the non-Archi\-me\-dean property) implies the total disconnectedness of $\Q$ in the topology
determined by the metric $|x-y|_p$, as well as many unusual geometric properties. Note also the following consequence of the ultra-metric inequality: $|x+y|_p=\max (|x|_p,|y|_p)$, if $|x|_p\ne |y|_p$. We denote $\mathbb Z_p=\{ x\in \Q:\ |x|_p\le 1\}$. $\mathbb Z_p$, as well as all balls in $\Q$, is simultaneously open and closed.(such sets are called clopen). A characteristic function of a clopen set is continuous; moreover, it is an example of a locally constant function, that is a function constant on a neighborhood of each point. The set $\D(\Q,\mathbb C)$ of locally constant functions $\Q\to \mathbb C$ with compact supports (with an appropriate topology; see \cite{VVZ}) is used as a space of test functions in $p$-adic harmonic analysis. Below we use also the similar space $\D(\Q,\Q)$ of $p$-adic-valued functions.

Denote by $m(dx)$ the Haar measure on the
additive group of $\Q$ normalized by the equality $\int_{\mathbb Z_p}m(dx)=1$.

\section{Infinite dimensional  $p$-adic  affine group}

Consider a function $b:\Q \to \Q$, $b\in \D(\Q,\Q)$, that is compactly supported locally constant on $\Q$.
Take another function $a:\Q\to \Q$, such that  $a(x)=1$ outside a compact set $K$,   $a(x)\neq 0$, $ a|_{K}$
is locally constant.
For such functions we have the following representations: there exists a compact subset $K\subset \Q$, such that
$K=\cup_1^N B_k$ is the disjoint  union of balls and
\begin{equation}
\label{b}
b(x)= \sum_1^N b_k 1_{B_k}(x),\;\;b_k\in\Q,\;\;x\in\Q,
\end{equation}
\begin{equation}
\label{a}
a(x) = \sum_1^N a_k 1_{B_k}(x) + 1_{K^c}(x), \;\;a_k \in \Q,\;\; x\in \Q,
\end{equation}
where $K^c$ denotes the complement of the set $K$.
Introduce an infinite dimensional $p$-adic affine group
$\Aff$ as the set of all pairs $g=(a,b)$ with components satisfying the above assumptions. Define the group
operation
$$
g_1 g_2=  (a_2,b_2) (a_1, b_1) = (a_1 a_2, b_2 +a_2 b_1).
$$
The unity in this group is $e=(1,0)$.
For $g\in \Aff$  we have $g^{-1}= (a^{-1}, -ba^{-1})$.

\vspace{2mm}

For $x\in\Q$ consider the section $G_x= \{g(x)\; |\; g\in \Aff\}$.
It is an affine group with constant coefficients. Note that for a ball $B_N (0) \subset \Q$ with the radius $p^N$
centered at zero we have $g(x)= (1,0), x\in B^c_N(0)$.

Define the action of $g$ on a point $x\in\Q$ as
$$
gx= g(x)x = \frac{x+b(x)}{a(x)},
$$
$$
\Q \ni x\mapsto gx\in \Q.
$$
Denote $O_x=\{gx| g\in G_x\}\subset \Q$.
For any element $y\in O_x$  and $h\in G_x$ we can define
$hy= h(gx)= (hg)x\in O_x$. It means that we have the group
$G_x$ action on the orbit $O_x$.

It gives

$$
(g_1g_2)(x) x= g_1(x)( g_2(x)x)
$$
 that corresponds to the group multiplication
 $$
 g_1 g_2=  (a_2,b_2) (a_1, b_1) = (a_1 a_2, b_2 +a_2 b_1)
 $$
 considered in the given point $x$.

\begin{Remark}
The situation we have is quite different from the case of the standard
group of motions on a phase space. Namely,
we have one fixed point $x\in\Q$ and the section group
$G_x$  associated with this point. Then we have the motion of
$x$ under the action of $G_x$. It gives the group action on the
orbit $O_x$.
\end{Remark}

\vspace{2mm}

We will work with the configuration space $\Ga(\Q)$; see \cite{AKR,VGG75}.

Each configuration may be identified with the measure
$$
\gamma(dx) = \sum_{x\in\gamma} \delta_x
$$
which is a positive  Radon measure on $\Q$: $\gamma\in \M(\Q)$.
We will use the vague topology on $\Ga(\Q)$ that is the image of the vague topology on the
space of positive Radon measures $\M(\Q)$. This topology is the weakest of those, for which all the mappings
$$
\Ga(\Q)\ni \ga \mapsto <f.\ga> =\int_{\Q} f(x) \ga(dx)\in \R
$$
are continuous for all $f\in \D(\Q,\R)$.

For $\ga\in \Ga(\Q)$, $\ga=\{x\}\subset \Q$ define
$g\gamma$ as a motion of the measure $\ga$:

$$
g\ga=\sum_{x\gamma} \delta_{g(x)x}\in \M(\Q).
$$
Here we have the group action of $\Aff$ produced by individual transformations
of points from the configuration. Again, as above, we move a fixed configuration using
previously defined actions of $G_x$ on $x\in\ga$.

Note that $g\gamma$ is not necessarily a configuration. More precisely, for some $B_N(0) $
the set $(g\ga)_N= g\ga\cap B_N^c(0)$ is a configuration in $B^c_N(0) $ but the finite part
of $g\ga$ may include multiple points. Therefore, we cannot consider an action of
$\Aff$ inside of the configuration space $\Ga(\Q)$. But we can define an action of this group
on certain class of functions on $\Ga(\Q)$.

For any $f\in \D(\Q,\C)$ we have the corresponding cylinder  function on $\Ga(\Q)$:
$$
L_f(\ga)=  <f,\ga > = \int_{\Q} f(x)\ga(dx) = \sum_{x\in \ga} f(x).
$$
Denote $\PC$ the set of all  cylinder polynomials generated by such functions.
More generally, consider functions of the form

\begin{equation}
\label{cyl}
F(\ga)= \psi(<f_1,\ga>,\dots, <f_n,\ga>),\; \ga\in\Ga(\Q),  f_j\in \D(\Q), \psi\in C_b(\R^n).
\end{equation}

These functions form the set $\CF_b(\Ga(\Q))$ of all  bounded cylinder functions.

For any clopen set $\Lambda \in \mathcal{O}_b(\Q)$ (also called a finite volume) denote
$\Ga(\Lambda)$ the set of all (with necessity finite) configurations
in $\La$.   We  have as before the vague topology on this space, and
the Borel $\sigma$-algebra $\B(\Ga(\La))$ is generated by functions
$$
\Ga(\La)\ni\ga \mapsto <f,\ga>\in\R
$$
for $f\in C_0 (\La)$.  For any $\La\in \mathcal{O}_b(\Q)$ and $T\in \B(\Ga(\La))$
define a cylinder set
$$
C(T)=\{\ga\in\Ga(\Q)\;|\; \ga_{\La}=\ga \cap \La \in T\}.
$$
Such sets form a $\sigma$-algebra $\B_{\La}(\Ga(\Q))$ of cylinder sets
for the finite volume $\La$. We denote by $B_{\La}(\Ga(\Q))$ the set of bounded functions on $\Ga(\Q)$ measurable
with respect to $\B_{\La}(\Ga(\Q))$  . That is a set of cylinder functions
on $\Ga(\Q)$.  As a generating family for this set we can use the functions of the form
$$
F(\ga)= \psi(<f_1,\ga>,\dots, <f_n,\ga>),\; \ga\in\Ga(\X),  f_j\in C_0(\La), \psi\in C_b(\R^n).
$$

For so-called one-particle functions $f:\Q\to\R, f\in\D(\Q)$, consider
$$
(gf)(x)= f(g(x) x), x\in \Q.
$$
 Then $gf\in \D(\Q)$. Thus,
 we have the group action
 $$
 \D(\Q)\in f \mapsto gf\in \D(\Q),\;\;g\in\Aff
 $$
 of the infinite dimensional group $\Aff$ in the space of functions
 $\D(\Q)$.

Note that due to our definition, we have
$$
<f, g\ga> = <gf,\ga>
$$
and it is reasonable to define for cylinder functions the action of the group $\Aff$
as
$$
(V_g F)(\ga)= \psi(<gf_1,\ga>,\dots <gf_n,\ga>.
$$
Obviously $V_g: \CF_b (\Ga(\Q))\to \CF_b(\Ga(\Q))$.

The dual transformation to one-particle motion is defined
via the following relation
$$
\int_{\Q} f(g(x)x) m(dx)=\int_{\Q} f(x) g^\ast m(dx)
$$
if there exists such measure $g^\ast m$ on $\Q$. As before, $m(dx)$ is the Haar measure on $\Q$.

\begin{Lemma}
\label{gm}

For each $g\in \Aff$
$$
g^\ast m(dx)=  \rho_{g}(x)  m(dx)
$$
where $\rho_g  = 1_{B_R^c(0) } + r_g^0,\;\; r_g^0\in \D(\Q,\R_+).$
Here as above
$$
B_R^c(0)= \{x\in\Q\;|\; |x|_p > R\}.
$$

\end{Lemma}

\begin{proof}
We have following representations for coefficients of $g(x)$:

$$
b(x)= \sum_{k=1}^{n} b_k 1_{B_k}(x) ,
$$
$$
a(x)= \sum_{k=1}^{n} a_k 1_{B_k}(x) + 1_{B^c_R(0)}(x)
$$
where $B_k$ are certain balls in $\Q$, see (\ref{a}), (\ref{b}).
Then
$$
\int_{\Q} f(g(x)x) m(dx)= \sum_{k=1}^n \int_{B_k} f(\frac{x+b_k}{a_k}) m(dx) + \int_{B^c_R (0)} f(x) m(dx) =
$$
$$
\sum_{k=1}^{n} \int_{C_k} f(y) |a_k|_p m(dy) + \int_{B^c_R(0)} f(y) m(dy),
$$
where
$$
C_k= a_k^{-1}(B_k + b_k).
$$
Therefore,
$$g^\ast m= (\sum_{k=1}^n  |a_k|_p 1_{C_k} + 1_{B^c_R(0)}) m.
$$
Note that informally we can write
$$
(g^\ast m)(dx) = dm(g^{-1}x).
$$
(compare with a general formula of an analytic change of variables from \cite{VVZ}).
\end{proof}

Note that by the duality we have the group action on the Haar measure. Namely,
for $f\in \D(\Q)$ and $g_1, g_2\in \Aff$
$$
\int_{\Q} (g_2 g_1) f(x)  m(dx)= \int_{\Q} g_1 f (x)  (g_2^\ast m) (dx) =
$$
$$
\int_{\Q} f(x) (g_1^\ast  g_2^\ast m)(dx)= \int_{\Q} f(x) ((g_2 g_1)^\ast m)(dx).
$$
In particular
$$
(g^{-1})^\ast (g^\ast m)= m.
$$

\begin{Lemma} Let $F\in B_\La (\Ga(\Q))$, and $g\in\Aff $ has the form
$g(x)=(1, h1_{B}(x))$ with certain $h\in \Q$ and $B\in \mathcal{O}_b(\Q)$ such that $\La\subset B$.
Then
$$
V_gF\in B_{\La -h} (\Ga(\Q)).
$$

\end{Lemma}
\begin{proof}
Due to the formula for the action $V_gF$ we need to analyze the support
of functions $f_j (x+h1_B(x))$ for $\supp f_\subset \La$. If $x\in B^c$ then
$x\in \La^c$ and therefore $f_j (x+h1_B(x))=f_j(x)=0$. For $x\in B$
we have $f_j(x+h)$ and only for $x+h\in \La$ this value may be nonzero,
i.e., $\supp g f_j \subset \La- h$.

\end{proof}

\begin{Lemma}
\label{V}
For all $F \in \PC$ or $F\in \CF_b (\Ga(\Q))$ and $g\in \Aff $, the formula
$$
\int_{\Ga(\Q)} V_g F d\pi_m = \int_{\Ga(\Q)} Fd\pi_{g^\ast m} .
$$
holds.

\end{Lemma}

\begin{proof}
It is enough to show this equality for exponential functions
$$
F(\ga)= e^{<f,\ga>},\;\; f\in\D(\Q).
$$

We have
$$
\int_{\Ga(\Q)} V_g F d\pi_m = \int_{\Ga(\Q)} e^{<gf, \ga>} d\pi_m(\ga)=
$$
$$
\exp[ \int_{\Q} (e^{gf(x)} -1) dm(x)] = \exp[ \int_{\Q} (e^{f(x)} -1) d(g^{\ast} m)(x)=
$$
$$
\int_{\Ga(\Q)} F d\pi_{g^\ast m }.
$$

\end{proof}

\begin{Remark} For all functions $F,G\in \CF(\Ga(\Q))$ a similar
calculation shows
$$
\int_{\Ga(\Q)} V_g F  \; Gd\pi_m = \int_{\Ga(\Q)} F  \; V_{g^{-1}} G d\pi_{g^\ast m} .
$$
\end{Remark}
Let $\pi_m$ be the Poisson measure on $\Ga(\Q)$ with the intensity
measure $m$. For any $\La\in \mathcal{O}_b(\Q)$ consider the distribution $\pi_m^\La$
of $\pi_m$ in $\Ga(\La)$ corresponding to the projection $\ga\to \ga_\La$.
It is again a Poisson measure $\pi_{m_\La}$ in $\Ga(\La)$ with the intensity
$m_\La$ which is the restriction of $m$ on $\La$.  Infinite divisibility of
$\pi_m$ gives for $F_j\in B_{\La_j}(\Ga(\Q)), j=1,2$ with $\La_1\cap \La_2=\emptyset$ that
$$
\int_{\Ga(\Q)}  F_1(\ga) F_2(\ga) d\pi_m(\ga)= \int_{\Ga(\Q)}  F_1(\ga) d\pi_m(\ga)
\int_{\Ga(\Q)}  F_2(\ga) d\pi_m(\ga)=
$$
$$
\int_{\Ga(\La_1)} F_1 d\pi^{\La_1}_m \int_{\Ga(\La_2)} F_2 d\pi^{\La_2}_m.
$$

\begin{Lemma}

For any $F\in B_\La(\Ga(\Q)$ and $g=(1, h1_B)\in \Aff $ with $\La \cap (B+h)=\emptyset$ holds
$$
\int_{\Ga(\Q)} (V_g F)(\ga) d\pi_m(\ga)= \int_{\Ga(\Q)} F(\ga)d\pi_m(\ga).
$$

\end{Lemma}

\begin{proof}
Due to our calculations above we have
$$
\int_{\Ga(\Q)} (V_gF)(\ga) d\pi_m(\ga)= \int_{\Ga(\Q)} F(\ga) d\pi_{g^{\ast}m}(\ga)=
$$
$$
\int_{\Ga(\La)} F(\eta) d\pi^{\La}_{g^{\ast}m} (\eta) =\int_{\Ga(\La)} F(\eta) d\pi_{ (g^{\ast}m)_\La} (\eta).
$$
But we have shown
$$
(g^{\ast}m)(dx)= (1+ 1_{B+h}(x)) m(dx) = m(dx)
$$
for $x\in \La$, i.e.,  $(g^{\ast}m)_\La =m$.

\end{proof}

\begin{Lemma}
\label{prod}
For any $F_1,F_2 \in \CF_b(\Ga(\Q))$ there exists $g\in\Aff$ such that
$$
\int_{\Ga(\Q)} F_1 \; V_g F_2  d\pi_m = \int_{\Ga(\Q)} F_1 d\pi_m  \int_{\Ga(\Q)} F_2 d\pi_m .
$$

\end{Lemma}

\begin{proof}
By the definition, $F_j\in B_{\La_j}(\Ga(\Q)), j=1,2$ for some $\La_1,\La_2 \in \mathcal{O} (\Q)$.

 Let us take $g=(1, h1_B)$ with the following assumptions:
 $$
 \La_2\subset B,\;\; \La_1\cap (\La_2-h) =\emptyset,\;\; \Lambda_2\cap (B+h) =\emptyset.
 $$
 Then according to previous lemmas
$$
\int_{\Ga(\Q)} F_1 V_g F_2 d\pi_m = \int_{\Ga(\Q)} F_1 d\pi_m  \int_{\Ga(\Q)} F_2 d\pi_m .
$$

\end{proof}

\section{$\Aff$ and Poisson measures}

For $F\in \PC $  or $F\in \CF_b (\Ga(\Q))$, we consider the motion of $F$ by $g\in \Aff$ given by
the operator $V_g$.
Operators $V_g$   have the group property
defined point-wisely: for any $\ga \in \Ga(\Q) $

$$
(V_h (V_gF))(\ga)= (V_{hg} F) (\ga).
$$
This equality is the consequence of  our definition of the group action of $\Aff$
on cylinder functions.

As above,
consider  $\pi_m$, the Poisson measure on $\Ga(\Q)$ with the intensity measure $m$. For the transformation
$V_g$ the dual object is defined as the measure $V^\ast_g \pi_m$ on $\Ga(\Q)$ given by the relation
$$
\int_{\Ga(\Q)} (V_gF) (\ga) d\pi_m(\ga) =\int_{\Ga(\Q)} F(\ga)  d(V^\ast_g \pi_m)(\ga),
$$
where $V^\ast_g \pi_m= \pi_{g^\ast m}$, see Lemma \ref{V}.

\begin{Corollary}
For any $g\in \Aff$ the  Poisson measure $V_g^\ast \pi_m$ is absolutely continuous
 w.r.t. $\pi_m$  with the Radon-Nykodim derivative
$$
R(g,\ga)= \frac{d\pi_{g^\ast m}(\ga)}{d\pi_{ m} (\ga)} \in L^1(\pi_m).
$$.

\end{Corollary}

\begin{proof}
Note that density $\rho_g  = 1_{B_R^c(0) } + r_g^0,\;\; r_g^0\in \D(\Q,\R_+)$ of $g^\ast m$ w.r.t. $m$
may be equal zero on some part of $\Q$ and, therefore, the equivalence of of considered
Poisson measures is absent. Due to \cite{LS03}, the Radon-Nykodim derivative
$$
R(g,\ga)= \frac{d\pi_{g^\ast m}(\ga)}{d\pi_{ m} (\ga)}
$$
exists if
$$
\int_{\Q} |\rho_g(x)-1| m(dx)= \int_{B_R(0)} |1-r_g^0 (x)| m(dx) <\infty.
$$
\end{proof}

\begin{Remark}
As in the proof of Proposition 2.2 from \cite{AKR} we have an explicit formula for $R(g,\ga)$:

$$
R(g,\ga)= \prod_{x\in\ga} \rho_g (x) \exp(\int_{\Q} (1-\rho_g(x)) m(dx).
$$
The point-wise existence of this expression is obvious.

\end{Remark}

This fact gives us the possibility to apply the Vershik-Gelfand-Graev approach realized by these authors for the case
of diffeomorphism group.

Namely, for $F\in \PC$ or $F\in \PC(\Ga(\Q)$ and $g\in \Aff$ introduce operators
$$
(U_g F)(\ga) = (R(g^{-1} ,\ga) )^{1/2} (V_gF)(\ga).
$$

\begin{Theorem}

Operators   $U_g,\; g\in \Aff$ are unitary in $L^2 (\Ga(\Q), \pi_m)$ and give an irreducible representation
of $\Aff$.

\end{Theorem}

\begin{proof}
Let us check the isometry property  of these operators.
We have using Lemmas \ref{V}, \ref{gm}
$$
\int_{\Ga(\Q)} |U_g|^2 d\pi_m = \int_{\Ga(\Q)} |V_g F|^2(\ga) d\pi_{(g^{-1})^\ast m} (\ga)=
$$
$$
\int_{\Ga(\Q)} |F(\ga)|^2 d\pi_{(gg^{-1})\ast m}(\ga)= \int_{\Ga(\Q)} |F(\ga)|^2 d\pi_{ m}(\ga).
$$
From Lemma \ref{V} follows that $U_g^\ast = U_{g^{-1}}.$

We need only to check irreducibility that shall
follow from the ergodicity of Poisson measures \cite{VGG75}. But to this end we need first of all to define the action of
 the group $\Aff$ on sets from $\B(\Ga(\Q)$.  As we pointed out above, we can not define this
 action point-wisely. But we can define the action of     operators $V_g$ on the indicators $1_A(\ga)$ for
 $A\in \B(\Ga(Q))$. Namely, for given $A$ we take a sequence of cylinder sets $A_n, n\in \N$ such that
 $$
 \pi_{m}(A\Delta A_n) \to 0, n\to \infty.
 $$
 Then
 $$
 U_g 1_{A_n} =V_g 1_{A_n} (R(g^{-1} ,\cdot) )^{1/2}  \to G (R(g^{-1} ,\cdot) )^{1/2}  \in L^2(\pi_m), n\to\infty
 $$
 in $L^2(\pi_m)$. Each $V_g 1_{A_n} $ is an indicator of a cylinder set and
 $$
 V_g 1_{A_n}  \to G \;\; \pi_m - a.s., n\to \infty.
 $$
 Therefore,
 $G=1$ or $G=0$ $\pi_m$-a.s. We denote this function $V_g 1_A$.

 For the proof of the ergodicity of the measure $\pi_m$ w.r.t. $\Aff$ we need to show the following fact:
 for any $A\in \B(\Ga(\Q))$ such that $\forall g\in\Aff\;\; V_g 1_A = 1_A\; \pi_m- a.s.$ holds $\pi_m(A)= 0$
 or $\pi_m(A)= 1$.

 Fist of all, we will show that for any pair of sets $A_1, A_2 \in \B(\Ga(Q))$ with $\pi_m(A_1)>0,\;\;
 \pi_m(A_2) >0$ there exists $g\in\Aff$ such that
 \begin{equation}
 \label{ineq}
 \int_{\Ga(\Q)} 1_{A_1} V_g 1_{A_2} d\pi_m \geq \frac{1}{2} \pi_m(A_1) \pi_m(A_2).
 \end{equation}
 Because any Borel set may be approximated by cylinder sets, it is enough to show this fact
 for cylinder sets. But for such sets due to Lemma \ref{prod}  we can choose $g\in \Aff$ such that
$$
 \int_{\Ga(\Q)} 1_{A_1} V_g 1_{A_2} d\pi_m =  \pi_m(A_1) \pi_m(A_2).
 $$
Then using an approximation we will have (\ref{ineq}).

 To finish the proof of the ergodicity, we consider any $A\in\B(\Ga(\Q)$  such that
 $$
 \forall g\in \Aff\; V_g1_A = 1_A \;\;\pi_m - a.s.,\;\; \pi_m(A)>0.
 $$
 We will show that then $\pi_m(A)= 1$.  Assume $\pi_m(\Ga\setminus A) >0$.
 Due to the statement above, there exists $g\in \Aff$ such that
 $$
 \int_{\Ga(\Q)} 1_{\Ga\setminus A} V_g 1_A >0.
 $$
 But due to the invariance of $1_A$  it means
 $$
 \int_{\Ga(\Q)} 1_{\Ga\setminus A} 1_A d\pi_m >0
 $$
 that is impossible.
\end{proof}

\subsection*{Acknowledgement}

The work of the first-named
author was funded in part by the budget program of Ukraine No. 6541230
``Support to the development of priority research trends''. It was
also supported in part under the research work ``Markov evolutions
in real and p-adic spaces'' of the Dragomanov National
Pedagogical University of Ukraine.


\begin{thebibliography}{999}
\bibitem{AJO}
H. Airault, S. Jendoubi, and H. Ouerdiane, Unitarising measures for the representations of affine group and associated invariant operators. {\it Bull. Sci. Math.} {\bf 137} (2013), 775--790.
\bibitem{AKR}
S.~Albeverio, {Yu}.~G. Kondratiev, and M.~R{\"o}ckner. Analysis and geometry on configuration spaces. {\it J.~Funct.~Anal.}, {\bf 154} (1998), 444--500.
\bibitem{AK}
S. Albeverio and S. V. Kozyrev, Frames of $p$-adic wavelets and orbits of the affine group. {\it p-Adic Numbers Ultrametric Anal. Appl.} {\bf 1} (2009), no. 1, 18--33.
\bibitem{Ar}
R. Artzy, {\it Linear Geometry}, Addison-Wesley, Reading, 1965.
\bibitem{EH}
A. S. Elmabrok and O. Hutnik, Induced representations of the affine group and intertwining operators: I. Analytical approach.
{\it J. Phys. A} {\bf 45} (2012), no. 24, 244017, 15 pp.
\bibitem{GJ}
V. Gayral and D. Jondreville, Quantization of the affine group of a local field. {\it J. Fractal Geom.} {\bf 6} (2019), 157--204.
\bibitem{Is}
R.S. Ismagilov, {\it Representations of Infinite-Dimensional Groups}. Translations of Mathematical
Monographs 152, American Mathematical Society, Providence, RI, 1996.
\bibitem{Jo}
D. Jondreville, A locally compact quantum group arising from quantization of the affine group of a local field. {\it Lett. Math. Phys.} {\bf 109} (2019), 781--797.
\bibitem{Ko}
A. Kosyak, {\it Regular, Quasi-regular and Induced Representations of Infinite-Dimensional Groups}, European Mathematical Society, Z\"urich, 2018.
\bibitem{LS03}
M. A. Lifshits and E. Yu. Shmileva, Poisson measures that are quasi-invariant with respect to multiplicative transformations.  {\it Theory Probab. Appl.} {\bf 46} (2003), 652--666.
\bibitem{Ly}
R. Lyndon, {\it Groups and Geometry}, Cambridge University Press, 1985.
\bibitem{VGG75}
A. M. Vershik, I. M. Gel'fand, M. I. Graev, Representations of the group of diffeomorphisms, {\it Russian Math. Surveys}, {\bf 30}, no.6 (1975), 1--50.
\bibitem{VVZ}
V. S. Vladimirov, I. V. Volovich and E. I. Zelenov, {\it $p$-Adic Analysis and Mathematical Physics}, World Scientific, Singapore, 1994.
\bibitem{Ze}
A. M. Zeitlin, Unitary representations of a loop $ax+b$ group, Wiener measure and $\Gamma$ -function, {\it J. Funct. Anal.} {\bf 263} (2012), 529--548.
\end{thebibliography}
\end{document}